\documentclass[12pt,leqno]{amsart}
\usepackage{graphicx}
\usepackage{soul}
\usepackage{amsmath,amssymb,amsfonts,mathrsfs, centernot}
\usepackage{amsthm}
\usepackage{color}
\usepackage{enumerate}
\usepackage{xfrac}
\usepackage{hyperref}
\usepackage{comment}
\usepackage{wrapfig}
\usepackage{float}

\makeatletter
\def\paragraph{\@startsection{paragraph}{4}%
  \z@\z@{-\fontdimen2\font}%
  {\normalfont\bfseries}}
\makeatother

\usepackage{tikz}
\usetikzlibrary{matrix,arrows,calc,fit,cd,positioning,intersections,arrows.meta, 3d,decorations.pathmorphing}

\newtheorem{introthm}{Theorem}

\newtheorem{theorem}{Theorem}[section]

\newtheorem{lemma}[theorem]{Lemma}
\newtheorem{proposition}[theorem]{Proposition}

\theoremstyle{definition}

\newtheorem{remark}[theorem]{Remark}

\newcommand{\N}{\mathbb{N}}

\newcommand{\R}{\mathbb{R}}

\begin{document}
\pagebreak


\title{Isometries of the Ebin metric}

\author{David Lenze}

\address
  {Karlsruher Institut f\"ur Technologie\\ Fakult\"at f\"ur Mathematik \\
Englerstr. 2 \\
76131 Karlsruhe,
Germany}
\email{david.lenze@kit.edu}

\begin{abstract}  We study the space of Riemannian metrics over a compact manifold equipped with the Ebin metric. We characterize its self-isometries and prove that two such spaces are isometric if and only if their underlying manifolds are diffeomorphic.
\end{abstract}

\maketitle

\renewcommand{\theequation}{\arabic{section}.\arabic{equation}}
\pagenumbering{arabic}
\section{Introduction}

\subsection{Main results}
Let $M$ be an $n$-dimensional compact and smooth manifold, and let $\mathcal M$ denote the infinite dimensional space of smooth Riemannian metrics on $M$. In \cite{MR267604}, Ebin introduced a natural $L^2$-type Riemannian structure on $\mathcal{M}$, now known as the \emph{Ebin metric}:
\[(h,k)_{g} = \int \operatorname{tr}(g^{-1}hg^{-1}k)dV_g,\]
where $g\in \mathcal M$, $h,k\in T_g\mathcal M \cong \Gamma(S^2 \, T^\ast M)$. Here, $g^{-1}h$ is interpreted as the $(1,1)$-tensor derived from $h$ by raising an index using $g^{-1}$.

Freed and Groisser (\cite{MR1027070}) and later Gil-Medrano and Michor (\cite{MR1107281}) in more generality investigated the curvature and geodesics of the resulting space by means of explicitly determining structures like the Levi-Civita connection and the curvature tensor of this infinite-dimensional metric -- each requiring a theory of infinite-dimensional Riemannian geometry to make sense of these formal calculations.

In contrast to this approach, a new metric perspective emerged through the work of Clarke \cite{MR3010152}, who showed that the Ebin metric induces a distance $d_E$ on $\mathcal M$ (a non-trivial result in the infinite-dimensional context, cf. \cite{MR2148075, MR2201275} for examples of vanishing induced distances), that the resulting metric space is incomplete, and that its completion is a CAT$(0)$ space.

Recently Cavallucci in \cite{C} provided a shorter, more conceptual proof of a strengthened result: the completion is CAT$(0)$ and, surprisingly, its isometry class \textit{depends only on the dimension of the underlying manifold}. Thus, analyzing the completion's geometry yields no information about the underlying manifold beyond its dimension.

In sharp contrast to Cavallucci's finding, we show that the isometry class of the uncompleted space $(\mathcal M(M),d_E)$ depends on $M$ in the strongest plausible way:
\begin{introthm}\label{main}
	Let $M$ and $N$ be compact and smooth manifolds. $(\mathcal M(M),d_E)$ is isometric to $(\mathcal M(N),d_E)$ if and only if $M$ is diffeomorphic to $N$. 
\end{introthm}

To set the stage for our study of the Ebin metric, we begin by recalling the structure of the space of Riemannian metrics in more detail:
let $E:=E(M)$ denote the fibre bundle $\bigsqcup_{p\in M} S^2_+(T_pM)$ of symmetric, positive definite $(0,2)$-tensors on $M$, where the fibre is the cone of symmetric, positive definite $n\times n$ matrices, denoted $P(n)$. The space of Riemannian metrics on $M$ is precisely the space of smooth sections of $E$, i.e. $\mathcal M=\Gamma(E)$.  

Now we fix a background metric $g_0$ on $M$, and observe that: \begin{align*} (h,k)_{g} = \int \operatorname{tr}(g^{-1}hg^{-1}k)\sqrt{\det(g_0^{-1}g)} dV_{g_0} .\end{align*} The expression $\langle a,b \rangle_{x,p}:= \operatorname{tr}(x^{-1}ax^{-1}b)\sqrt{\det(g_0(p)^{-1}x)}$ defines a Riemannian metric on the fibres $E_p$ for each $p\in M$, where $x\in E_p$ and $a,b\in T_{x}E_p \cong S^2(T_pM)$, and we denote by $d_p$ the distance induced by this metric. 
In short, the Ebin metric is obtained by integrating the fibre-wise Riemannian metrics $\langle h(p),k(p)\rangle_{g(p),p}$ over $M$. 

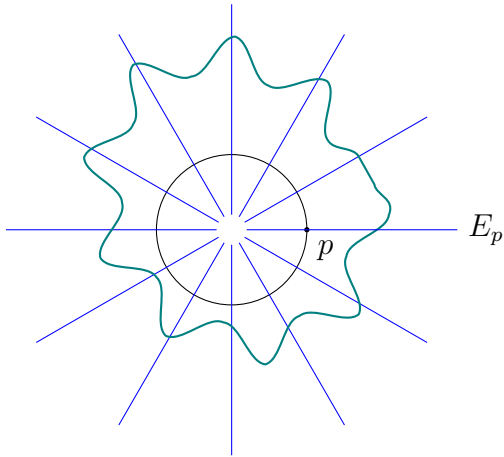
\begin{figure}[b]\label{E}
    \begin{tikzpicture}
    \def\R{1} 
    \def\RayStart{0.2}
    \def\RayEnd{3}

    \def\CenterX{0.25} 
    \def\CenterY{0}     
    \def\XRadius{1.7} 
    \def\YRadius{2}
    \def\RotationAngle{30} 

    \def\WobbleAmplitude{7pt} 
    \def\WobbleSegmentLength{35pt} 
    \draw[] (0,0) circle(\R cm);
    
    \foreach \angle in {0, 30, 60, 90, 120, 150, 180, 210, 240, 270, 300, 330 } {
        \ifdim\angle pt=0pt\else
            \draw[blue] (\angle:\RayStart) -- (\angle:\RayEnd);
        \fi
    }
    
    \draw[teal, thick, 
        decorate, 
        decoration={snake, segment length=\WobbleSegmentLength, amplitude=\WobbleAmplitude},
        rotate=\RotationAngle, shift={(\CenterX,\CenterY)}] 
        (0,0) ellipse (\XRadius cm and \YRadius cm)
        ;

    \fill (\R, 0) circle (1pt) node[anchor=north west] {$p$}; 

    \draw[blue] (0:\RayStart) -- (0:\RayEnd) 
        node[black, anchor= west] {$E_p$}; 
\end{tikzpicture}
  \caption{\small Fibre bundle $E$ over the base manifold $M$. The manifold is represented by the black circle; a generic smooth section, i.e. Riemannian metric $g\in \Gamma(E)$ by the green curve, and typical fibres by the blue rays.}
\end{figure}

This integrated structure extends to the distance: by \cite[Theorem 3.8]{MR3010152} (see also \cite[Theorem B.1]{MR4624071} for a simplified proof), the distance induced by the Ebin metric is an $L^2$-type integral over the fibre-wise distances. More specifically, up to a choice of normalization: \[d_E(g_1,g_2)=\left(\int_M d_p^2(g_1(p),g_2(p)) \, dV_{g_0}(p)\right)^{\frac{1}{2}}. \]
We emphasize that while the metrics $d_p$ on the fibres depend on the choice of the background metric $g_0 \in \mathcal{M}(M)$, this dependence is precisely balanced by the integration measure $dV_{g_0}$ such that the resulting integrated expression is independent of $g_0$.

The spaces $(E_p,d_p)$ are incomplete and isometric to $P(n)$ endowed with the metric $\operatorname{tr}(x^{-1}ax^{-1}b)\sqrt{\det(x)}$, where $x\in P(n)$ and $a,b\in T_xP(n)\cong S(n)$, and we denote the induced metric by $d^\prime_{P(n)}$.

The disjoint union $\text{Isom}(E):=\bigsqcup_{p\in M} \text{Isom}(E_p)$ forms a smooth fibre bundle with fibre the Lie group $G:=\text{Isom}(P(n),d^\prime_{P(n)})$; the space of smooth sections $\Gamma(\text{Isom}(E))$ forms a group by fibre-wise composition.

Both $\text{Diff}(M)$ and $\Gamma(\text{Isom}(E))$ are subgroups of the isometry group of $\mathcal M(M)$ with respect to the Ebin metric:

\begin{enumerate}
	\item Let $\varphi \in \text{Diff}(M)$, then by \cite{MR267604} the pullback $g\mapsto \varphi^\ast g$ is an isometry of $(\mathcal M(M),d_E)$. 
	\item Let $\alpha \in \Gamma(\text{Isom}(E))$, then the map $g\mapsto \alpha (g)$ is an isometry of $(\mathcal M(M),d_E)$, where $\alpha(g)_p:=\alpha_p(g_p)$.
\end{enumerate}

We show that the isometries of $(\mathcal M(M),d_E)$ are rigid, arising as compositions of isometries of the above types. Indeed:

\begin{introthm}\label{isom}
	Let $M$ be a compact and smooth manifold. Then $$\text{Isom}(\mathcal M(M),d_E)=\Gamma(\text{Isom}(E(M)))\rtimes\text{Diff}(M).$$
\end{introthm}

These results should be compared to recent work of Darvas \cite{MR4321209} on the isometries of the space of K\"ahler metrics with respect to the \textit{Mabuchi metric}. There it was shown that all isometries  are induced by biholomorphisms and anti-biholomorphisms of the manifold.

\subsection{General Strategy}

The proofs of our main results will depend on a rigidity result for the \emph{isometries of the completion} of the space of Riemannian metrics with respect to the Ebin metric.
Before stating it, we recall the explicit description of the completion offered in \cite{C}.

Let $\overline {P(n)}$ denote the metric completion of $P(n)$ with respect to $d^\prime_{P(n)}$.
We extend the fibre bundle $\pi_E: E\to M$ to a fibre bundle $\pi_{\overline E}: \overline E \to M$ with fibre $\overline{P(n)}$. Each fibre $\overline E_p$, $p\in M$, is equipped with a metric $d_p$ such that $(\overline E_p, d_p)$ is isometric to $(\overline{P(n)}, d_{P(n)}^\prime)$ and there exists an inclusion of fibre bundles $E \hookrightarrow \overline E$ which, when restricted to the fibres, is an isometric embedding $(E_p,d_p) \hookrightarrow (\overline E_p,d_p)$ with dense image.

\begin{figure}[t]\label{E_bar}
    \begin{tikzpicture}
    \def\R{1} 
    \def\RayStart{0.2}
    \def\RayEnd{3}

    \def\CenterX{0.25} 
    \def\CenterY{0}     
    \def\XRadius{1.7} 
    \def\YRadius{2}
    \def\RotationAngle{30} 

    \def\WobbleAmplitude{7pt} 
    \def\WobbleSegmentLength{35pt} 
    \draw[] (0,0) circle(\R cm);
    
    \foreach \angle in {0, 30, 60, 90, 120, 150, 180, 210, 240, 270, 300, 330 } {
        \ifdim\angle pt=0pt\else
            \draw[blue, opacity=0.5] (\angle:\RayStart) -- (\angle:\RayEnd);
        \fi
    }
    
    \draw[teal, opacity=0.9, 
        decorate, 
        decoration={snake, segment length=\WobbleSegmentLength, amplitude=\WobbleAmplitude},
        rotate=\RotationAngle, shift={(\CenterX,\CenterY)}] 
        (0,0) ellipse (\XRadius cm and \YRadius cm)
        ;

    \fill (\R, 0) circle (1pt) node[anchor=north west] {$p$}; 

    \draw[blue, opacity=0.9] (0:\RayStart) -- (0:\RayEnd) 
        node[black, anchor= west] {$\overline E_p$}; 
        
    \draw[blue, thick] (0,0) circle(\RayStart cm);
    
    \draw[red,  thick] 
        (20:0.9) .. controls (40:2.1) and (60:2.4) .. (80:1.1)
        
        -- (120:1.6)
        
        ; 
    
    \draw[red, thick] 
        (120:1.2) 
        
        -- (160:1.8) 
        -- (190:1.3)
        
        -- (220:1.8)
        
        ; 
        
    \draw[red, thick, 
        decorate, 
        decoration={snake, segment length=10pt, amplitude=2pt}]
        (220:1.8) -- (250:1.4); 
        
    \draw[red, thick]
        (250:1.1) .. controls (270:0.2) and (300:2.4) .. (330:0.8);
        
     \draw[red, thick]
        (330:0.8) .. controls (0:0.1) and (0:0.05) .. (20:0.6);

\end{tikzpicture}

  \caption{\small The \emph{completed} bundle $\overline E$. The smaller blue circle represents the elements added to complete the fibres (blue rays) $E_p$ with respect to $d_p$ -- thereby forming new fibres $\overline E_p$. The red piecewise continuous curve represents an $L^2$-section $\sigma \in \mathcal E(M)$. Compare with Figure~1.} 
\end{figure}
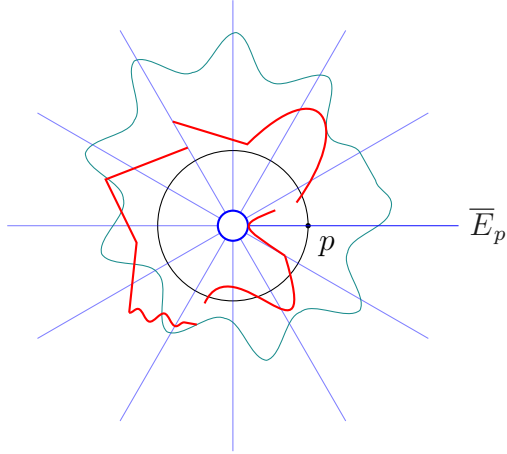

An $L^2$-section of this bundle is defined as a section $\sigma:M\to \overline E$ which is such that $\int_M d_p^2(\sigma(p),g_0(p)) \,dV_{g_0}(p)<\infty$. The space of $L^2$-sections of $\overline E$ is denoted by $\mathcal E(M)$ and contains the space of smooth Riemannian metrics, i.e. $\mathcal M(M) =\Gamma(E) \subset \mathcal E(M)$. We equip it with the so-called $L^2$-metric which naturally extends the Ebin metric: $$d_{L^2}(\sigma,\sigma^\prime):=\left(\int_Md_p^2(\sigma(p),\sigma^\prime(p))\,dV_{g_0}(p) \right)^{\frac{1}{2}}.$$ 

In \cite{C}, Cavallucci showed that the metric completion of the space $(\mathcal M(M),d_E)$ can be identified with the space $(\mathcal E(M),d_{L^2})$ of $L^2$-sections of the extended bundle $\overline {E(M)}$. 

Note that two measures $\mu$ and $\nu$ on a measurable space $\Omega$ are \emph{equivalent}, denoted $\mu \sim \nu$, if they are mutually absolutely continuous, i.e. $\mu \ll \nu$ and $\nu \ll \mu$. With this, we can finally state the promised result:

\begin{introthm}\label{isom_completion}
		Let $M$ be a compact, smooth manifold with a fixed Riemannian metric $g_0$. If $\gamma: (\mathcal E(M),d_{L^2}) \to (\mathcal E(M),d_{L^2})$ is an isometry, then there exist an almost everywhere bijective $\varphi:M\to M$ with $\varphi_\ast\mu_{g_0} \sim \mu_{g_0}$, and a family $\{\tau_p\}_{p\in M}$ of dilations $\tau_p:(\overline{E_{\varphi(p)}},d_{\varphi(p)}) \to  (\overline{E_p},d_p)$, such that for all $\sigma\in \mathcal{E}(M)$, we have $\gamma(\sigma)_p=\tau_p(\sigma_{\varphi(p)})$.
\end{introthm}

The key idea behind the proof of Theorem~\ref{isom} in Section~\ref{proof_main} relies on the fact that any isometry $\gamma:(\mathcal M(M),d_E) \to (\mathcal M(M),d_E)$ uniquely extends to an isometry $\overline\gamma:(\mathcal E(M),d_{L^2}) \to (\mathcal E (M),d_{L^2})$ of the completion, and therefore inherits the rigidity established in Theorem~\ref{isom_completion}: there exists an almost everywhere bijective $\varphi:M\to M$ with $\varphi_\ast\mu_{g_0} \sim \mu_{g_0}$ and a family $\{\tau_p\}_{p\in M}$ of dilations $\tau_p:(E_{\varphi(p)},d_{\varphi(p)}) \to  (E_p,d_p)$, such that $\gamma(g)_p=\tau_p(g_{\varphi(p)})$. 

We exploit the fact that $\gamma(g)\in \mathcal{M}(M)$ to show that $\varphi$ is almost everywhere equal to a diffeomorphism, allowing us to assume it is a diffeomorphism from the start. 

As mentioned above, by \cite{MR267604}, the pullback $g\mapsto\varphi^\ast g$ is an isometry with respect to the Ebin metric, and so $\gamma \circ (\varphi^\ast)^{-1}$ is also an isometry. By the above rigidity, this will have the form $(\gamma \circ (\varphi^\ast)^{-1})(g)_p=\alpha_p(g_p)$, with $\alpha:M \to \text{Isom}(E)$ a section of the smooth fibre bundle $\text{Isom}(E) =\bigsqcup_{p\in M} \text{Isom}(E_p,d_p)$. We show that this section is smooth and so $\text{Isom}(\mathcal{M}(M),d_E)= \Gamma(\text{Isom}(E)) \cdot \text{Diff}(M)$. The proof concludes by noting that the intersection $\text{Diff}(M)\cap \Gamma(\text{Isom}(E))$ is trivial and that $\Gamma(\text{Isom}(E))$ is a normal subgroup. Thus, we have the semi-direct product: $\text{Isom}(\mathcal{M}(M),d_E)= \Gamma(\text{Isom}(E)) \rtimes \text{Diff}(M). $

To prove Theorem~\ref{main}, we will similarly show that, given an isometry $\gamma:\mathcal{M}(M) \to \mathcal{M}(N)$, there exists a diffeomorphism $\psi:N \to M$, and a family $\{\kappa_p\}_{p\in N}$ of dilations $\kappa_p:E_{\psi(p)}(M) \to E_p(N)$ such that $\gamma(g)_p=\kappa_p(g_{\psi(p)})$. Our result follows as a direct consequence.

To set the background needed to understand the strategy behind the proof of Theorem~\ref{isom_completion}, we briefly recall the notion of $L^2$ spaces with metric space targets: 

The space $L^2(\Omega,X)$ consists of measurable, essentially separably valued maps $f:\Omega \to X$ from a finite measure space $(\Omega,\mu)$ to a metric space $(X,d)$ so that 
$
\int_\Omega d^2(f(\omega),x)\,d\mu(\omega)<\infty,
$
for some (and hence any) $x\in X$. 
The space is naturally equipped with the metric 
\[
d_{L^2}(f,g) = \left(\int_\Omega d^2(f(\omega),g(\omega))\, d\mu(\omega)\right)^{1/2}. \]

The formal similarity in the definitions of $\mathcal E (M)$ and  $L^2(\Omega,X)$ lies at the heart of Cavallucci's argument to show that the $L^2$-completion of the space of Riemannian metrics is CAT$(0)$. Indeed Cavallucci showed that $(\mathcal E(M),d_{L^2})$ is isometric to $L^2(M,\overline{P(n)})$, with $\overline{P(n)}$ the metric completion of $(P(n),d_{P(n)}^\prime)$ from above. The assertion then follows from the fact that $\overline{P(n)}$ is CAT$(0)$, and the fact that $L^2$-spaces with CAT$(0)$ targets are themselves CAT$(0)$ spaces. 

We will follow the same philosophy for the proof of Theorem~\ref{isom_completion}, by showing a corresponding rigidity for the isometries of $L^2(M,\overline{P(n)})$ which we will translate in Section~\ref{proof_isom_c} to the isometries of $(\mathcal E(M),d_{L^2})$.

Before clarifying this further, recall that a map $f: X \to Y$ between metric spaces is \emph{affine} if it sends geodesics to linearly reparametrized geodesics, and is a \textit{dilation} if the reparametrization factor is independent of the geodesic. A metric space $X$ is \textit{affinely rigid} if every affine map from $X$ to any metric space $Y$ is a dilation. $X$ is \textit{affinely rigid with respect to $\mathrm{CAT}(0)$ spaces} if this holds for all affine maps into any $\mathrm{CAT}(0)$ space $Y$. 

We will show in the Preliminaries that $\overline{P(n)}$ is a Euclidean cone over a geodesically complete CAT$(0)$ space, and in Section~\ref{cone} that such cones are affinely rigid.

 The promised rigidity for for the isometries of $L^2(M,\overline{P(n)})$ is thus a special case of the following considerably more general result of independent interest: 

\begin{introthm}\label{L^2isometry}
	Let $(\Omega,\mu)$ be a standard probability space and let $X$ be a separable, geodesically complete CAT$(0)$ space which is affinely rigid with respect to CAT$(0)$ spaces and not isometric to $\R$. A map $\gamma:L^2(\Omega,X)\to L^2(\Omega,X)$ is an isometry if and only if there exist an a.e. bijective $\varphi:\Omega\to \Omega$ with $\varphi_\ast \mu \sim \mu$, and an a.e. defined map $\rho:\Omega \to \operatorname{Dil}(X)$ with $\operatorname{dil}(\rho(\omega))=\left(\sqrt{\frac{d(\varphi_\ast \mu)}{d\mu}}\right)^{-1}$ and $\rho(\cdot)(x)\in L^2(\Omega,X)$ for some $x\in X$, such that $\gamma(f)(\omega)=\rho(\varphi(\omega))(f(\varphi(\omega))$ for $\mu$-a.e. $\omega \in \Omega$.
\end{introthm}

This result is analogous to \cite[Theorem E]{isom} for $L^2$-spaces with Riemannian manifold targets, yet a key difference is that the Riemannian case yields a family of isometries, while this setting allows for dilations, reflecting \emph{greater isometric flexibility}. 
 
The proof strategy remains the same as in \cite{isom}: we first \emph{characterize affine maps on metric products of $X$}, and then extend this to $L^2(\Omega,X)$. This is used to establish that any splitting $L^2(\Omega,X)=Y \times \overline Y$ is canonical (i.e.  $L^2(A,X)\times L^2(A^c,X)$ up to an isometry of the factors,
 for some measurable $A\subset \Omega$) which will be the key to analyzing the isometries in the remaining argument.

The main difference from \cite{isom} is in the characterization of affine maps on products (Lemma \ref{orig}, the analogue of \cite[Theorem 3.1]{isom}). Unlike the Riemannian approach in \cite{isom}, which relied on holonomy and Lytchak's work on affine maps on Riemannian manifolds \cite{affineimages}, our proof is \emph{fundamentally metric}, drawing on the work of Foertsch and Lytchak \cite{MR2399098} and CAT$(0)$ geometry.

\section{Preliminaries}

\subsection{Metric spaces} Let $X:=(X,d)$ be a metric space.
A \emph{geodesic} in $X$ is an isometric embedding of an interval into $X$, it is called a segment if the interval is compact. 
The metric space $X$ is a \emph{geodesic metric space} if every pair of points
in $X$ is joined by a geodesic.

A complete geodesic metric space is a CAT$(0)$ space if, for any geodesic triangle, the distance between any two points on its sides is less than or equal to the distance between the corresponding points on the sides of its comparison triangle in the Euclidean plane $\mathbb E^2$. See \cite{BH} for more details. 

A map $f\colon X \to Y$ between metric spaces is \textit{affine} if it maps geodesics to linearly reparametrized geodesics, with a reparametrization factor that may depend on the geodesic. An affine map is called a \textit{dilation} if the reparametrization factors are independent of the geodesics. 

The space of dilations $f:X \to X$ with factor $\lambda\geq 0$  is denoted $\text{Dil}_\lambda(X)$, and the space of all dilations $\text{Dil}(X)=\bigsqcup_{\lambda\geq 0} \text{Dil}_\lambda(X)$.  

Examples of affine maps that are not dilations are projections from metric products onto their factors: by \cite[I.5.3]{BH}, for a geodesic $\gamma=(\gamma_1,\gamma_2)$ in a metric product $X_1\times X_2$, $\gamma_1$ and $\gamma_2$ are linearly reparametrized geodesics in $X_1$ and $X_2$ respectively, showing that the projection maps $X_1\times X_2 \to X_1$ and $X_1\times X_2 \to X_2$ are affine. 
Another class of examples arises from normed vector spaces through a change of norm: linear segments $t\mapsto v+t\cdot w$ are linearly reparametrized geodesics in $(\R^n,|\cdot|)$ for any norm $|\cdot|$, and $(\R^n,|\cdot|)$ is uniquely geodesic if and only if $|\cdot|$ is strictly convex. Thus for a strictly convex norm $|\cdot|$ and any other norm $|\cdot|^\prime$, the identity map $(\R^n,|\cdot|)\to (\R^n, |\cdot|^\prime)$ is affine. As long as $|\cdot|^{\prime}$ is not just a simple multiple of $|\cdot|$, the resulting affine map is not a dilation. 

We call a metric space $X$ \textit{affinely rigid} if all affine maps $f\colon X \to Y$, into any metric space $Y$, are dilations. We say that $X$ is \textit{affinely rigid with respect to CAT$(0)$ spaces} if all affine maps $f\colon X \to Y$, into any CAT$(0)$ space $Y$, are dilations. In the previous project \cite{MR4958894}, the author studied the property of affine rigidity in the context of geodesically complete CAT$(0)$ spaces admitting geometric group actions.

\subsection{The space of positive definite symmetric matrices} Let $P(n)\subset \R^{n\times n}$ be the space of positive-definite symmetric $n\times n$ matrices, which as an open cone in the vector space $S(n)\subset \R^{n\times n}$ of symmetric $n\times n$ matrices naturally carries the structure of a smooth manifold. 

Let $x\in P(n)$, and $a,b\in T_xP(n) \cong S(n)$, then $$g_x(a,b):=\text{tr}(x^{-1}ax^{-1}b),$$ defines a Riemannian metric on $P(n)$, the induced distance of which we denote by $d_{P(n)}$. The Riemannian manifold $(P(n),g)$ is well-studied and has many interesting properties, of which we recall a few: \begin{enumerate}
	\item $(P(n),g)$ is a symmetric space and can be identified with the homogeneous space $GL(n,\R)/O(n)$.
	\item $(P(n),d_{P(n)})$ is a CAT$(0)$ space.
	\item The space $P_1(n)\subset P(n)$ of positive-definite symmetric $n\times n$ matrices with determinant one is a totally geodesic submanifold of $(P(n),g)$, and as such also CAT$(0)$. $P_1(n)$ is also a symmetric space (as the homogeneous space $SL(n,\R)/SO(n)$).
	\item $P(n)$ splits as $\R \times P_1(n)$ via the isometry $(s,p) \mapsto e^{\frac{s}{\sqrt{n}}}p$. 
	\end{enumerate}
We refer to \cite[Chapter II.10]{BH} for more details. 

For the present purpose, we consider another Riemannian metric $g^\prime$ on $P(n)$ obtained via a conformal change of $g$ with factor $x\mapsto \sqrt{\det(x)}$: $$g_x^\prime(a,b)=g_x(a,b)\sqrt{\det(x)}=\text{tr}(x^{-1}ax^{-1}b)\sqrt{\det(x)},$$ where again $x\in P(n)$ and $a,b\in T_xP(n) \cong S(n)$.

In analogy to the product decomposition of $(P(n),g)$, $(P(n),g^\prime)$ is isometric to the \textit{warped product} $(\R \times P_1(n), (e^{\frac{t\sqrt{n}}{2}}dt^2) \times_{e^{\frac{t\sqrt{n}}{4}}} g)$ (cf. \cite{C}). Since $(\R,e^{\frac{t\sqrt{n}}{2}}dt^2)$ is isometric to the open ray $(\R_{>0},dt^2)$ via $t\mapsto \frac{4}{\sqrt{n}}e^{\frac{t\sqrt{n}}{4}}$, the space $(P(n),g^\prime)$ is also expressible as a warped product with linear warping function: $(P(n),g^\prime)\cong (\R_{>0} \times P_1(n), dt^2 \times_{\frac{t\sqrt{n}}{4}} g)$.

The metric $d_{P(n)}^\prime$ induced by $g^\prime$ on $P(n)$ is incomplete, as demonstrated by Cauchy sequences of the form $(\frac{1}{n},x)_{n\in \N}$ for fixed $x\in P_1(n)$. Further, the space is dense in the complete \emph{metric warped product} $(\R_{\geq 0} \times P_1(n), d_{eucl}\times_\frac{t\sqrt{n}}{4} d_{P(n)})$ (for the definition of metric warped products with possibly vanishing warping functions, see \cite{MR4734965}).

 Thus, $(\R_{\geq 0} \times P_1(n), d_{eucl}\times_{\frac{t\sqrt{n}}{4}} d_{P(n)})$ is the completion of $(P(n), d_{P(n)}^\prime)$. Since the warping function is linear, the metric warped product is a  Euclidean cone \cite[Chapter 11.C]{MR4734965}. More specifically: $(\R_{\geq 0} \times P_1(n), d_{eucl}\times_{\frac{t\sqrt{n}}{4}} d_{P(n)}) \cong C_0(P_1(n),\frac{\sqrt{n}}{4}d_{P(n)})$, and so:
\begin{lemma}\label{cone_P}
	  	The metric completion of $(P(n),d_{P(n)}^\prime)$ is isometric to the Euclidean cone $C_0(P_1(n),\frac{\sqrt{n}}{4}d_{P(n)})$. 
\end{lemma}

Since it is isometric to a Euclidean cone over a CAT$(0)$ space, the completion of $(P(n),d_{P(n)}^\prime)$ is also a CAT$(0)$ space (cf. \cite[II. 3.14]{BH}).

\subsection{Measure spaces}

We introduce some measure-theoretic concepts underlying our work. We refer to Bogachev \cite[Chapter 9]{MR2267655}. 

Let $(\Omega_1,\mathcal A_1, \mu_1)$ and $(\Omega_2,\mathcal A_2,\mu_2)$ be measure spaces. 
\begin{enumerate}
\item A bijection $\varphi:\Omega_1 \to \Omega_2$ with $A\in \mathcal A_1 \iff \varphi(A) \in \mathcal A_2$ and $\mu_2(\varphi(A))=\mu_1(A)$ for all $A\in \mathcal A_1$ is a \textit{strict isomorphism}.
\item A map $\varphi:\Omega_1 \to \Omega_2$ for which there exist $N_1\in \mathcal A_1$ and $N_2 \in \mathcal A_2$ with $\mu_1(N_1)=\mu_2(N_2)=0$, and such that $\Omega_1\setminus N_1 \to \Omega_2 \setminus N_2$, $\omega \mapsto \varphi(\omega)$ is a strict isomorphism is called an \textit{isomorphism}. An isomorphism $\varphi:\Omega \to \Omega$ is an \textit{automorphism}. 
\end{enumerate}

A map $\varphi:\Omega_1 \to \Omega_2$ for which there exist null sets $N_1\subset \Omega_1$ and $N_2\subset \Omega_2$ such that $\Omega_1\setminus N_1 \to \Omega_2 \setminus N_2$, $\omega \mapsto \varphi(\omega)$ is bijective is called \textit{almost everywhere bijective}.

A measure space $(\Omega,\mu)$ is a \textit{standard probability space} or \textit{Lebesgue-Rokhlin space} if it is isomorphic to $[0,1]\sqcup \N$, equipped with the measure $\nu:=c\lambda+\sum_{n=1}^{\infty}p_n\delta_n$, where $c\geq 0$ and $p_n\geq 0$ for all $n\in \N$. 
Here $\lambda$ denotes the Lebesgue measure on $[0,1]$ and $\delta_{n}$ the Dirac measure at  $n \in \N$.  The atoms are represented by the integers and the space is atomless if $p_n=0$ for all $n\in \N$. Of course we have that $c+\sum_{n=0}^\infty p_n=1$.

\subsection{$L^2$-spaces with metric space targets}

We briefly recall the definition and some properties of $L^2$-spaces with metric space targets. For more details, see e.g., \cite{KS, Monod, isom}.

Let $(\Omega, \mu)$ be a finite measure space, and $(X,d)$ a metric space. The space $L^2(\Omega,X)$ consists of measurable functions $f:\Omega\to X$ (up to null-sets) with separable range that satisfy 
$$\int_\Omega d^2(f(g),x)d\mu(g)<\infty$$
for some (and thus any) $x\in X$. We naturally equip $L^2(\Omega,X)$ with the metric $d_{L^2}(f,f^\prime):=\left(\int_\Omega d^2(f(g),f^\prime(g))d\mu(g)\right)^{\frac{1}{2}}$.

For a fixed target $(X,d)$, the isometry class of $L^2(\Omega,X)$ depends only on the measure space class of $\Omega$. Specifically, if $(\Omega_1, \mu_1)$ and $(\Omega_2, \mu_2)$ are isomorphic measure spaces with isomorphism $\varphi:\Omega_1 \to \Omega_2$, the map $f\mapsto f\circ \varphi$ defines an isometry between $L^2(\Omega_2,X)$ and $L^2(\Omega_1,X)$.

The geodesics in $L^2(\Omega,X)$ are characterized by the following result:
\begin{theorem}(\cite[Proposition 44]{Monod})\label{thm: monodgeodesics_concise}
A continuous map $\sigma:I\to L^2(\Omega,X)$ (where $I\subset \R$ is an interval) is a geodesic if and only if there exists a measurable function $\alpha:\Omega \to \R_{\geq 0}$ such that $\int_\Omega \alpha(\omega)^2 d\mu(\omega)=1$, and a collection of geodesics $\{\sigma^\omega\}_{\omega\in \Omega}$ in $X$ such that
$$\sigma(t)(\omega)=\sigma^\omega(\alpha(\omega)t) \text{ for all } t\in I \text{ and } \mu\text{-almost every } \omega\in \Omega.$$
\end{theorem}

Finally we record that $L^2$-spaces inherit many properties from their targets.
If $(\Omega, \mu)$ is a finite measure space, then:
\begin{enumerate}
    \item If $(X,d)$ is complete, $L^2(\Omega,X)$ is also complete.
    \item If $(X,d)$ is geodesic, then $L^2(\Omega,X)$ is also geodesic. 
    \item If $(X,d)$ is a CAT$(0)$ space, $L^2(\Omega,X)$ is also a CAT$(0)$ space. \end{enumerate}
 
\section{Affine maps on cones} \label{cone}

We begin by recording the following auxiliary result, which is a direct  corollary of \cite[Proposition 3.3]{MR2399098}. 

\begin{lemma}\label{affine_rectangle}
	Let $I,J\subset \R$ be two closed intervals, $Z$ a geodesic metric space and $f:I\times J \to Z$ an affine map. Parallel lines in $I\times J$ are reparametrized uniformly under $f$: let $a,b\in \R$ and $(t_1,t_2)\in I\times J$ -- the speed of $s\mapsto f(as+t_1,bs+t_2)$ is independent of $(t_1,t_2)\in I\times J$.
\end{lemma}

\begin{proof}
	By \cite[Proposition 3.3]{MR2399098}, there exists an isometric embedding $f(I\times J) \hookrightarrow V$ into a normed space $V$ such that the affine map $I\times J\xrightarrow{(t,s)\mapsto f(\gamma(t),\eta(s))} f(I\times J) \hookrightarrow V$ sends geodesics in $I\times J$ to linear geodesics in $V$. As described in \cite{MR2399098}, this map can be extended to an affine map $F:\R^2 \to V$ in the classical sense, the restriction of which to $I\times J$ coincides with the above map $I\times J \to V$. Since the classical affine map $F$ sends parallel lines in $\R^2$ to parallel lines in $V$, we deduce that the reparametrization factors are independent as claimed.
\end{proof}

\begin{figure}[h]
    \centering
    \begin{tikzpicture}[scale=1, thick]

        \def\rectwidth{3.0}
        \def\rectheight{1.5}
        \def\shearfactor{1.0}
        \def\xshift{5.5} %
        
        \coordinate (A) at (0, 0);
        \coordinate (B) at (0, \rectheight);
        \coordinate (C) at (\rectwidth, \rectheight);
        \coordinate (D) at (\rectwidth, 0);
         
         \draw[draw=blue, line width=1pt] (A) -- (B) -- (C) -- (D) -- cycle;
        \node at (\rectwidth/2, -0.4) {$I\times J$}; %
        
        \coordinate (A_prime) at (\xshift + 0 + \shearfactor * 0, 0);
        \coordinate (B_prime) at (\xshift + 0 + \shearfactor * \rectheight, \rectheight);
        \coordinate (C_prime) at (\xshift + \rectwidth + \shearfactor * \rectheight, \rectheight);
        \coordinate (D_prime) at (\xshift + \rectwidth + \shearfactor * 0, 0);
    
        \draw[draw=blue, line width=1pt] (A_prime) -- (B_prime) -- (C_prime) -- (D_prime) -- cycle;
        \node at (\xshift + \rectwidth/2 + \shearfactor * \rectheight/2, -0.4) {$f(I\times J)\subset Z$}; 
        
        \coordinate (Start) at (\rectwidth + 0.5, \rectheight/2);
        \coordinate (End) at (\xshift - 0.25, \rectheight/2);
        
        \draw[->, black] (Start) to[out=20, in=160] (End);
      
       \node at (4.35,0) [yshift=0.5cm] {$f$};

   \draw[draw=teal, line width=2pt] (0.5, 0.2) -- (1, 1);
   \draw[draw=teal, line width=2pt] (1.2, 0.3) -- (1.7, 1.1);
   
   \draw[draw=teal, line width=2pt] (6.1, 0.2) -- (7.2, 0.7); 
   \draw[draw=teal, line width=2pt] (7.2, 0.3) -- (8.3, 0.8);

    \end{tikzpicture}
    \caption{\small The rectangle $I\times J$ is send to a (possibly degenerate) parallelogram. Parallel segments in $I\times J$ are send to parallel ones in $f(I\times J)$. }
\end{figure}
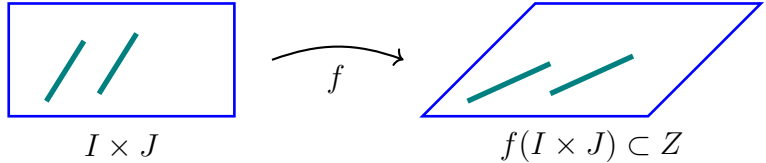

With this, we show that cones over geodesically complete CAT$(0)$ spaces are affinely rigid.

	\begin{theorem}\label{cone_prop_a}
		Let $X$ be a geodesically complete CAT$(0)$ space. Then the Euclidean cone $C_0(X)$ is affinely rigid. 
	\end{theorem}
	
	  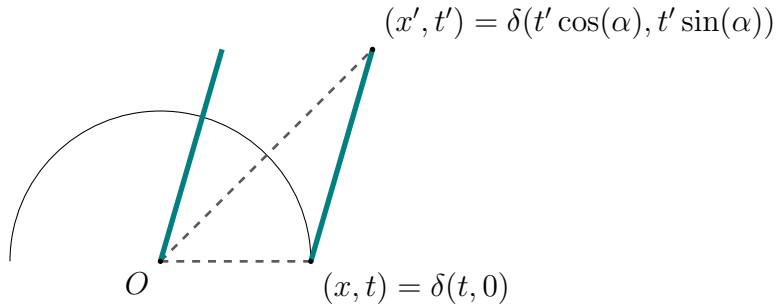
\begin{figure}[b]\label{cone_figure}
\begin{center}
\begin{tikzpicture}[scale=2] 

\pgfmathsetmacro{\R}{1} 
\pgfmathsetmacro{\Angle}{45}

\coordinate (X) at (\R, 0);

\coordinate (Y) at (1.41,1.41);

\coordinate (O) at (0, 0);

\draw[black] (\R, 0) arc (0:180:\R);

\draw[dashed, line width=1pt, gray!70!black] (O) -- (X); 
\draw[dashed, line width=1pt, gray!70!black] (O) -- (Y); 

\draw[teal, line width=2pt] (X) -- (Y);

\coordinate (Z) at ($(Y) - (X)$);

\draw[teal, line width=2pt] (O) -- (Z);

\fill (O) circle (0.5pt) node[below left] {$O$};

\fill (X) circle (0.5pt) node[below right] {$(x,t)=\delta(t,0)$};
\node at (1.1*\R, -0.1*\R) {}; 

\fill (Y) circle (0.5pt) node[above right] {$(x^\prime,t^\prime)=\delta(t^\prime \cos(\alpha),t^\prime \sin(\alpha))$};

\end{tikzpicture}
\caption{\small  Subcone $C_0(\eta([0,\pi]))$ isometric to the half plane $P_+\subset \mathbb E^2$. The black half-circle represents the segment $\eta_{[0,\pi]}$, the two parallel green segments are reparametrized uniformly under the affine map; compare Lemma~\ref{affine_rectangle} and Figure 3.}
\end{center}
\end{figure}

	\begin{proof}
		Let $f:C_0(X)\to Y$ be an affine map into some metric space $(Y,d_Y)$. We want to show that there exists $\lambda\geq 0$ such that for all $(t,x),(t^\prime,x^\prime)\in C_0(X)$, $$d_Y(f(t,x),f(t^\prime,x^\prime))=\lambda d_{C_0(X)}((t,x),(t^\prime,x^\prime)).$$ For any $x\in X$, $\gamma_x:[0,\infty)\to C_0(X)$, $t \mapsto (t,x)$ is a geodesic ray. 
Given $x,z \in X$ with $d(x,z)\geq \pi$, we notice that $\gamma_{xz}~:~\R\to C_0(X)$, defined by \[\gamma_{xz}(t):=\begin{cases} (-t,x) \text{ for } t\in (-\infty,0) \\(t,z) \text{ for } t\in [0,\infty)  
		
	\end{cases}\]
is a geodesic line. 

As a consequence of the geodesic completeness, given any two points $x,y\in X$, there exists $z\in X$ such that $d(x,z),d(y,z) \geq \pi$. Since the geodesics $\gamma_{xz}$ and $\gamma_{yz}$ both contain the geodesic $\gamma_z$, we conclude that $$\rho(\gamma_{xz})=\rho(\gamma_{yz})=\rho(\gamma_z)=:\lambda,$$ where  $\rho(\gamma)$ denotes the reparametrization factor of $\gamma$ under the affine map $f$.  Therefore we can deduce that for any $x\in X$ and $t,t^\prime\in [0,\infty)$, \begin{align*}\label{lambda}
		d_Y(f(t,x),f(t^\prime,x))=\lambda |t-t^\prime| =\lambda d_{C_0(X)}((t,x),(t^\prime,x)),
     \end{align*} and of course for any $x,x^\prime \in X$ with $d_X(x,x^\prime)\geq \pi$, $$d_Y(f(t,x),f(t^\prime,x^\prime))=\lambda (t+t^\prime)=\lambda d_{C_0(X)}((t,x),(t^\prime,x^\prime)). $$ Now assume we are given $x,x^\prime$ with $d_X(x,x^\prime) < \pi$ and any $t,t^\prime \geq 0$. There exists a geodesic connecting $x$ and $x^\prime$; by geodesic completeness it can be extended to a geodesic line $\eta:\R \rightarrow X$. Possibly after a linear reparametrization, we may assume that $\eta(0)=x$ and $\eta(d(x,x^\prime))=x^\prime$. Now let us consider the conical segment $C_0(\eta([0,\pi]))\subset C_0(X)$ over $\eta\vert_{[0,\pi]}$, and note that it is isometric to the upper half plane $P_+\subset \R^2$ via $$\delta:P_+ \rightarrow C_0(\eta([0,\pi])),(t\cos(\alpha),t\sin(\alpha)) \mapsto (t,\eta(\alpha)).$$ Thus we obtain an affine map $P_+ \xrightarrow{\delta}C_0(\eta([0,\pi]))\xrightarrow{f} f(C_0(\eta([0,\pi])))\subset Y$ from $P_+$ into $Y$. By Lemma~\ref{affine_rectangle}, the segment connecting $(t,0)\in P_+$ to $(t^\prime \cos(\alpha),t^\prime \sin(\alpha)) \in P_+$, where $\alpha=d(x,x^\prime)$, is reparametrized under this affine map by the same factor as the parallel segment issuing in the origin (see Figure~4 above). Since all geodesics issuing in the tip of the cone are reparametrized by $\lambda$, we have shown that indeed $d_Y(f(x,t),f(x^\prime,t^\prime))=\lambda d_{C_0(X)}((t,x),(t^\prime,x^\prime))$, completing the proof. 
 \end{proof}

\section{Affine maps on products}\label{affine_products}

In this section, we show that affine maps from metric products of $\text{CAT}(0)$ spaces that are affinely rigid and not isometric to $\mathbb{R}$ all arise from factor-wise reparametrizations and projections. More precisely:

\begin{lemma}\label{orig}
Let $X$ and $Y$ be geodesically complete CAT$(0)$ spaces, both affinely rigid with respect to CAT$(0)$ spaces, and let $Y$ be non-isometric to $(\mathbb{R},\|\cdot\|_2)$. Then, for any CAT$(0)$ space $Z$, every affine map $f: X \times Y \to Z$ is an isometry from $(X, \lambda d_X) \times (Y, \mu d_Y)$ to $f(X \times Y)$, for some $\lambda, \mu \geq 0$.
\end{lemma}

For the proof of this, we again make use of Lemma~\ref{affine_rectangle}.

\begin{proof}[Proof of Lemma~\ref{orig}]
	First notice that for any $y\in Y$, the map $(X,d_X)\to (Z,d_Z)$, $x\mapsto (x,y) \mapsto f(x,y)$ is affine. Thus, since $X$ is affinely rigid with respect to CAT$(0)$ spaces, the map is a dilation with uniform reparametrization factor $\lambda_y\geq 0$, i.e. for all $x,x^\prime \in X$, $$d_Z(f(x,y),f(x^\prime,y))=\lambda_y d_X(x,x^\prime).$$ Similarly, for every $x\in X$, there exists $\mu_x\geq 0$ so that for all $y,y^\prime \in Y$, $$d_Z(f(x,y),f(x,y^\prime))=\mu_x d_Y(y,y^\prime).$$
Take $(x_1,y_1), (x_2,y_2)\in X\times Y$ and geodesics $\gamma:I\to X$, $\eta:J\to Y$ connecting $x_1$ to $x_2$ and $y_1$ to $y_2$ respectively. Of course, $\gamma(I)\times \eta(J)$ is isometric to $I\times J$ via $(t,s) \mapsto (\gamma(t),\eta(s))$, and the map $I\times J \to Z$, $(t,s)\mapsto f(\gamma(t),\eta(s))$ is affine. 
By Lemma~\ref{affine_rectangle}, the parallel sides of $I\times J$ are reparametrized uniformly by this affine map and this gives us that $\lambda_{y_1}=\lambda_{y_2}=:\lambda$, and $\mu_{x_1}=\mu_{x_2}=:\mu$. In other words, we have shown that $f:(X,\lambda d_X)\times (Y,\mu d_Y) \to (Z,d_Z)$ is isometric restricted to each fibre of the product decomposition $X\times Y$.

We want to show that it is actually isometric on the entire product. If $\lambda=0$ or $\mu=0$ we are done. Thus we may assume that $\lambda,\mu >0$, and so by rescaling $\lambda=\mu=1$. Further, we can assume that $f$ is injective, for otherwise we can consider the map $f^\prime:(X,\sqrt{2}d_X)\times (Y,\sqrt{2}d_Y) \to (Z,d_Z)\times (X,d_X) \times (Y,d_Y)$, given by $(x,y) \mapsto (f(x,y),x,y)$, which is injective and still isometric restricted to each fibre. That this map is isometric implies that $f$ is isometric as well. Finally we can simplify to the case $X=\R$. Indeed the assumptions of being affine, injective, and isometric restricted to the fibres of the product decomposition carry over to the maps $\widetilde f_\gamma: \R\times Y \to Z$, $(t,y)\mapsto f(\gamma(t),y)$ for geodesics $\gamma:\R\to X$, and knowing that these maps are isometric for all such geodesics gives us the claim.  

Thus let $f:\R\times Y \to Z$ be an injective and affine map that is isometric restricted to each fibre of the product $\R\times Y$. By the uniqueness of geodesics in CAT$(0)$ spaces, the image $C=f(\R\times Y)$ is convex and thus CAT$(0)$, and the inverse $f^{-1}:C\to \R \times Y$ is affine. 

Define $\delta:=p_\R \circ f^{-1}  : C \to \R$ and $\epsilon:=p_Y \circ f^{-1}: C \to Y$ and observe that $x=f(\delta(x),\epsilon(x))$ for all $x\in C$.  
		
Let $c_y:\R \to C$, $y\in Y$, be given by $c_y(t):=f(t,y)$. For any distinct $y,y^\prime\in Y$, $c_y$ and $c_{y^\prime}$ are disjoint and parallel, and the union of the images of all these geodesic lines $c_y$, $y\in Y$ is equal to $C$. Let us fix some $y_0\in Y$, denote $c:=c_{y_0}$, and let $\pi:C\to c(\R)$ be the nearest point projection of $C$ onto the complete convex image of $c$. Denote $C^0:=\pi^{-1}(c(0))$. Then by \cite[II, 2.14.]{BH}, $C^0$ is convex and $j:\R \times C_0 \to C$, $(t,x) \mapsto f(\delta(x)+t,\epsilon(x))$ is an isometry. Now observe that as a composition of affine maps, the following is affine: $$\alpha: Y\xrightarrow{y\mapsto f(0,y)}C\xrightarrow{p_{1} \circ j^{-1}}  \R,$$ where $p_1$ denotes the projection onto the first factor of $\R \times C^0$. Since $Y$ is affinely rigid with respect to CAT$(0)$ spaces, $\alpha$ is a dilation, i.e. there exists $a\geq 0$ such that $|\alpha(y)-\alpha(y^\prime)|=a\cdot d_Y(y,y^\prime)$ for all $y,y^\prime \in Y$. 
If $a\neq 0$, $Y$ would be isometrically embedded into $\R$, in contradiction to the assumption that $Y$ is geodesically complete and not isometric to $\R$. Thus $a=0$, i.e. $\alpha$ is constant and we conclude that $\alpha(y)=\alpha(y_0)=0$ for all $y\in Y$. 
Thus $C^0=f(\{0\}\times Y)$ which is isometric to $Y$ via $\sigma:Y\to C_0$, defined by $y\mapsto f(0,y)$. Observe that for all $(t,y)\in \R\times Y$, we have $j\circ (\text{id} \times \sigma) (t,y)=f(p_\R\circ f^{-1}(f(0,y))+t,p_Y\circ f^{-1}(f(0,y)))=f(t,y)$, i.e. $j\circ (\text{id} \times \sigma) = f$ and thus we obtain the desired result.
\end{proof}

Applying the same ideas in an inductive argument we obtain: 

\begin{lemma}\label{lem_affine}
	Let $X$ be a geodesically complete CAT$(0)$ space, affinely rigid with respect to CAT$(0)$ spaces, and not isometric to $\R$. For a CAT$(0)$ space $Z$, every affine map $f:X^n\to Z$ is, up to factor-wise rescaling with some $c_i\geq 0$, an isometry $$(X,c_1d_X) \times ... \times (X,c_nd_X)\to f(X^n)\subset Z.$$
\end{lemma}

\section{Isometric localization and rigidity}

As described in the Introduction, the proof of Theorem~\ref{L^2isometry} follows the same line of reasoning as the proof of \cite[Theorem~E]{isom}.

We establish a characterization of affine maps analogous to \cite[Theorem 4.1]{isom}. By applying Lemma~\ref{lem_affine} (replacing \cite[Theorem 3.1]{isom}) and following the corresponding arguments in \cite{isom}, we obtain:

\begin{theorem}\label{affine_classical}
		Let $X$ be a geodesically complete CAT$(0)$ space which is affinely rigid with respect to CAT$(0)$ spaces and  not isometric to $\R$; and let $Y$ be a CAT$(0)$ space.  

		A Lipschitz map $F:L^2(\Omega,X)\to Y$ is affine if and only if there exists a nonnegative $\eta\in L^\infty(\Omega)$ such that, for all $f,g\in L^2(\Omega,X)$,
\[d_Y^2(F(f),F(g))=\int_\Omega\eta \,d_X^2(f,g)\, d\mu.\]\end{theorem}

Suppose $X$ is as in the assumptions of Theorem~\ref{affine_classical} and the space decomposes as a metric product: $L^2(\Omega,X)=Y\times \overline Y$. 

Since $Y\times \overline Y$ is CAT$(0)$ if and only both factors are CAT$(0)$, both $Y$ and $\overline Y$ are CAT$(0)$. The projections $P^Y:L^2(\Omega,X) \to Y$ and $P^{\overline Y}:L^2(\Omega,X) \to \overline Y$ are thus affine maps satisfying the assumptions of Theorem~\ref{affine_classical}. 

Therefore, there exist non-negative scaling functions $\eta,\overline \eta \in L^\infty(\Omega)$ such that for all $f,g\in L^2(\Omega,X)$, the squared distances in the projected spaces are given by: $$d_Y^2(P^Y(f),P^Y(g))=\int_\Omega\eta \,d_X^2(f,g)\, d\mu,$$ $$d_{\overline Y}^2(P^{\overline Y}(f),P^{\overline Y}(g))=\int_\Omega \overline\eta \,d_X^2(f,g)\, d\mu.$$

Finally, again arguing analogously to \cite{isom}, we see that there exists a measurable $A\subset \Omega$ such that $\eta=\chi_A$ and $\overline \eta = \chi_{A^c}$. This leads by the same arguments to an analogue of the \textit{isometric localization} result Lemma~5.2 in \cite{isom}:

\begin{proposition}\label{loc}
	Let $(\Omega,\mu)$ be a standard probability space and let $X$ be a geodesically complete CAT$(0)$ space which is affinely rigid with respect to CAT$(0)$ spaces and not isometric to $\R$. 
	
	Then, for any isometry $\gamma:L^2(\Omega,X) \to L^2(\Omega,X)$ there exists an almost everywhere bijective $\varphi:\Omega \to \Omega$ such that $\varphi_\ast \mu \sim \mu$, and for every measurable $A\subset \Omega$ and $f,g \in L^2(\Omega,X)$, \[\int_A d_X^2(f,g) \,d\mu= \int_{\varphi^{-1}(A)} d_X^2(\gamma(f),\gamma(g))d\mu.\]
\end{proposition}

\subsection{Isometric rigidity}

We are finally in a position to prove Theorem~\ref{L^2isometry} from the introduction.

\begin{proof}[Proof of Theorem~\ref{L^2isometry}]
	First assume that there exist $\varphi:\Omega \to \Omega$ and $\rho:\Omega \to \operatorname{Dil}(X)$ as above and that for almost all $\omega \in \Omega$, $\gamma(f)(\omega)=\rho(\varphi(\omega))(f(\varphi(\omega)).$ Observe that since $\rho(\cdot)(x)\in L^2(\Omega,X)$ for some $x\in X$, $\gamma(f)\in L^2(\Omega,X)$ for any $f\in L^2(\Omega,X)$. Furthermore, notice that \begin{align*}
		\int_\Omega d_X^2(\gamma(f),\gamma(g))d\mu&=\int_\Omega d_X^2(\gamma(f)\circ \varphi^{-1},\gamma(g)\circ \varphi^{-1})d(\varphi_\ast\mu)\\&=\int_\Omega d_X^2(\gamma(f)\circ \varphi^{-1},\gamma(g)\circ \varphi^{-1})\frac{d(\varphi_\ast\mu)}{d\mu}d\mu \\&=\int_\Omega d_X^2(f,g)d\mu.	\end{align*}
		Thus $\gamma:L^2(\Omega,X)\to L^2(\Omega,X)$ is an isometry as claimed, establishing one direction of the proof. 
		
		For the other direction, assume we are given an isometry $$\gamma:L^2(\Omega,X)\to L^2(\Omega,X).$$ By Proposition~\ref{loc}, there exists an almost everywhere bijective $\varphi:\Omega\to \Omega$ such that $\varphi_\ast \sim \mu$, and for all $A\subset \Omega$ and $f,g \in L^2(\Omega,X)$, \[\int_A d_X^2(f,g) \,d\mu= \int_{\varphi^{-1}(A)} d_X^2(\gamma(f),\gamma(g))d\mu.\]
		Since  $\varphi_\ast \mu \ll \mu$ and $\mu \ll \varphi_\ast \mu$, by Radon-Nikodym, there exists a $\mu$-a.e. positive Radon-Nikodym derivative $\delta:=\frac{d(\varphi_\ast\mu)}{d\mu}:\Omega\to \R$. 
		
		Thus, by a change of variable, for all $A\subset \Omega$ and $f,g\in L^2(\Omega,X)$, 
		\begin{align*}
		\int_A d_X^2(f,g)d\mu&=\int_{\varphi^{-1}(A)} d_X^2(\gamma(f),\gamma(g))d\mu\\&=\int_A d_X^2(\gamma(f)\circ \varphi^{-1},\gamma(g)\circ \varphi^{-1})\frac{d(\varphi_\ast\mu)}{d\mu}d\mu.\end{align*}
Hence, for all $f,g\in L^2(\Omega,X)$, we deduce that for $\mu$-a.e. $\omega\in \Omega$, \begin{align}\label{eqn}
	d_X(f(\omega),g(\omega))=\delta(\omega)^{\frac{1}{2}} \, d_X(\gamma(f)(\varphi^{-1}(\omega)),\gamma(g)(\varphi^{-1}(\omega))).
\end{align}

Fix a measure space isomorphism $\psi:(\Omega,\mu) \to ([0,1]\sqcup \N,c\lambda+\sum_{j\in \N}p_j\delta_j)$ with $c,p_j\geq 0$ and $c+\sum_{j\in \N}p_j=1$ (which exists because $\Omega$ is a standard probability space). 

For $x\in X^n$ and $y\in X^m$, we define a simple function $f_{x,y}:\Omega \to X$ by $f_{x,y}(\omega)=x_i$ for $\psi(\omega)\in [\frac{i-1}{n},\frac{i}{n})$ and $f_{x,y}(\omega)=y_j$ for $\psi(\omega)=j \leq m$ and $f_{x,y}(\omega)=y_m$ otherwise. Of course $f_{x,y}\in L^2(\Omega,X)$. For $x\in X=X^1$, we write $f_x:=f_{x,x}$ and this simply denotes the constant function. 

Since $X$ is separable, there exists a countable dense subset $D\subset X$. By \cite{isom}, the space $\mathcal S$ of all simple functions over finite measurable partitions is dense in $L^2(\Omega,X)$. In analogy to the case of the classical $L^2(\Omega)=L^2(\Omega,\R)$, it is easily seen that the countable subspace $$\mathcal D:=\{f_{x,y}: x,y\in \bigsqcup_{n\in \N} D^n\}\subset \mathcal S$$ is also dense. This recovers the fact that $L^2(\Omega,X)$ is separable.

Since countable intersections of full-measure subsets are of full measure, there exists a full-measure subset $\Omega^\prime \subset \Omega$ where equation~(\ref{eqn}) holds for all $f,g\in \mathcal D$.

This allows us to define the family of dilations $\rho:\Omega \to \text{Dil}(X)$ such that for $\mu$-a.e. $\omega\in \Omega$, $\gamma(f)(\omega)=\rho(\varphi(\omega))(f(\varphi(\omega))$ holds for all $f\in \mathcal D$.

Specifically, for $\omega\in \Omega^\prime$ and $x\in D$, define $\rho(\omega)(x):=\gamma(f_x)(\varphi^{-1}(\omega))$. Equation~(\ref{eqn}) then implies that $\rho(\omega)$ is a dilation on $D$:
\begin{align*} d_X(\rho(\omega)(x),\rho(\omega)(x^\prime)) &= d_X(\gamma(f_x)(\varphi^{-1}(\omega)),\gamma(f_{x^\prime})(\varphi^{-1}(\omega))) \\&= \delta(\omega)^{-\frac{1}{2}}\,d_X(x,x^\prime).\end{align*}
By continuity, $\rho(\omega)$ extends to a $\delta(\omega)^{-\frac{1}{2}}$-dilation on $X$.

Furthermore, for all $\omega\in \Omega^\prime$ and $f\in \mathcal D$, equation~(\ref{eqn}) yields 
$$d_X(\rho(\omega)(f(\omega)),\gamma(f)(\varphi^{-1}(\omega))) = \delta(\omega)^{-\frac{1}{2}}d(f(\omega),f(\omega))=0,$$
meaning $\rho(\omega)(f(\omega))=\gamma(f)(\varphi^{-1}(\omega))$.

Thus by substituting $\varphi(\omega)$, we see that for $\mu$-a.e. $\omega\in \Omega$, $\gamma(f)(\omega)=\rho(\varphi(\omega))(f(\varphi(\omega)))$ for all $f\in \mathcal D$.

Finally, by the density of $\mathcal D$ in $L^2(\Omega,X)$, this identity extends to all $f\in L^2(\Omega,X)$, proving the claim. 
\end{proof}

\section{Proof of Theorem~\ref{isom_completion}}\label{proof_isom_c}

\begin{proof}[Proof of Theorem~\ref{isom_completion}]
	Recall that by \cite[Theorem 3]{C}, the space $\mathcal E(M)$ is isometric to the space $L^2(M,\overline{P(n)})$. The isometry $\Psi:\mathcal E(M) \to L^2(M,\overline{P(n)})$ was constructed as follows: Take a triangulation of $M$, with $V_1,...,V_k$ the interiors of maximal dimensional simplices  which are contained in domains of trivializing bundle charts of $\overline{E}$ sending $\mu_{g_0}$ to the standard Euclidean volume form. For a section $\sigma: M \to \overline{E}$ in $\mathcal E(M)$, we define a function $f_{\sigma}:M \to \overline{P(n)}$ by locally composing $\sigma$ with the bundle trivialization and the projection onto the fiber. Specifically, on each $V_i$, $f_{\sigma}|_{V_i} := \pi_2 \circ \phi_i \circ \sigma|_{V_i}$, where $\phi_i:\pi^{-1}(V_i)\to V_i\times \overline{P(n)}$ is the local trivialization and $\pi_2$ is the projection on the second component. Since the complement of $\cup V_i$ is a null set, $f_{\sigma}$ is defined $\mu_{g_0}$-almost everywhere. The map $\Psi$ is then given by $\sigma \mapsto f_\sigma$, and by the fact that for every $p\in V_i$, $\Psi_p:=\pi_2\circ \phi_i\vert_{\overline{E}_p}:(\overline{E}_p,d_p)\to (\overline{P(n)},d_{P(n)}^\prime)$ is an isometry, we conclude that $\Psi$ is an isometry and given point-wise for a.e. $p\in M$ by $\Psi(\sigma)(p)=\Psi_p(\sigma_p)$.

	As a consequence, for an isometry $\gamma: \mathcal E(M) \to \mathcal E(M)$, the map $\Psi\circ \gamma \circ \Psi^{-1}:L^2(M,\overline{P(n)}) \to L^2(M,\overline{P(n)})$ is an isometry as well. 
	
	By Lemma~\ref{cone_P}, $\overline{P(n)}$ is isometric to the Euclidean cone $C_0(P_1(n))$. As a symmetric space of non-compact type $(P_1(n),d_{P(n)})$ is a geodesically complete CAT$(0)$ space and so, by Theorem~\ref{cone_prop_a}, $C_0(P_1(n))$ is affinely rigid. Finally, since geodesics branch in the tip, the space can not be isometric to $\R$. 
	
	Therefore Theorem~\ref{L^2isometry} is applicable and we know that $\Psi\circ \gamma \circ \Psi^{-1}$ is of the form $(\Psi\circ\gamma\circ \Psi^{-1})(f)(p)=\rho(\varphi(p))(f(\varphi(p)))$ for almost all $p\in M$, where $\varphi:M \to M$ is an almost everywhere bijective map with $\varphi_\ast \mu_{g_0} \sim \mu_{g_0}$, and $\rho:\Omega \to \operatorname{Dil}(X)$ a family of dilations with $\operatorname{dil}(\rho(p))>0$ for almost all $p\in M$. 
	
	Therefore, we conclude that for $\sigma \in \mathcal E(M)$ and almost all $p\in M$, \[\gamma(\sigma)_p=\underbrace{(\Psi^{-1}_{p}\circ \rho(\varphi(p))\circ \Psi_{\varphi(p)})}_{:=\tau_p \in \text{Dil}(E_{\varphi(p)},E_p)}(\sigma_{\varphi(p)}).\]
		 
\end{proof}

\section{Proofs of the main results}\label{proof_main}

\begin{remark}\label{remark}

	Recall that by \cite{MR267604}, the push-forward $g\mapsto \varphi^\ast g$ is an isometry of $\mathcal{M}(M)$ with respect to the Ebin metric. 
	
	We define $\tau^\varphi_p:E_{\varphi(p)}\to E_p$ by $\tau^\varphi_p(\sigma)(v,w)=\sigma(d\varphi_pv, d\varphi_pw)$, where $v,w\in T_pM$, and observe that $(\varphi^\ast g)_p=\tau^\varphi_p(g_{\varphi(p)})$. 
	
	As done before, by taking charts around $\varphi(p)$ and $p$ respectively that send $\mu_{g_0}$ to the standard Euclidean volume form, we obtain isometries $\Psi_{\varphi(p)}:(E_{\varphi(p)},d_{\varphi(p)})\to (P(n),d_{P(n)})$ and $\Psi_{p}:(E_{p},d_p)\to (P(n),d_{P(n)})$. The map $\Psi_p\circ \tau^\varphi_p \circ \Psi_{\varphi(p)}^{-1}:P(n) \to P(n)$ is of the form $P\mapsto J P J^T$, where $J \in GL(n,\R)$ represents $d\varphi_p$ with respect to the coordinate bases of the two charts from above; this map is a dilation with respect to $d_{P(n)}^\prime$ and therefore $\tau^\varphi_p:(E_{\varphi(p)},d_{\varphi(p)})\to (E_p,d_p)$ is a dilation as well. 

\end{remark}

\begin{proof}[Proof of Theorem~\ref{isom}]

 Every isometry $\gamma: \mathcal{M}(M) \to \mathcal{M}(M)$ extends uniquely to an isometry $\overline{\gamma}:\mathcal E(M) \to \mathcal E(M)$. By Theorem~\ref{isom_completion}, there exist an a.e. bijective $\varphi:M\to M$ and an a.e. defined family $\{\tau_p\}_{p\in M}$ of dilations $\tau_p:(\overline E_{\varphi(p)},d_{\varphi(p)}) \to  (\overline E_p,d_p)$ with $\text{dil}(\tau_p)>0$, such that \begin{align*}
 	\overline{\gamma}(\sigma)_p=\tau_p(\sigma_{\varphi(p)}).
 \end{align*}

The map $\overline \gamma$ is $\mathbb{R}^+$-homogeneous, satisfying $\overline \gamma(a\cdot \sigma)=a\cdot \overline \gamma(\sigma)$ for every $\sigma \in \mathcal{E}(M)$ and $a\geq 0$. This homogeneity is ensured because of the linear identifications $\Psi_p:\overline E_p \to \overline{P(n)}\cong C_0(P_1(n))$ for all $p\in M$, and the fact that any isometry of the cone $C_0(P_1(n))$ must preserve its tip. This preservation holds because, unlike the tip where geodesics branch, the remaining space $C_0(P_1(n))\setminus \{0\}\cong (P(n),g^\prime)$ is a Riemannian manifold in which geodesics never branch.

As a consequence, for all Riemannian metrics $g\in \mathcal{M}(M)$ and positive smooth functions $f\in C^\infty_+(M)$, we obtain $\overline\gamma(f\cdot g)=\hat f \cdot \overline \gamma(g)$, where $\hat f(p):=f(\varphi(p))$. Since by assumption, $\overline \gamma$ preserves $\mathcal{M}(M)$, both $\overline\gamma(g)\in \mathcal M(M)$ and $\overline\gamma(f\cdot g)\in \mathcal M(M)$. As a consequence, $\hat f \in C^\infty_+(M)$.

 Since both $\widehat{f+f^\prime}=\hat f + \hat{f^\prime}$ and $\widehat{ff^\prime}=\hat f \cdot \hat f^\prime$, the assignment $f\mapsto \hat f$ extends to a ring automorphism $\alpha:C^\infty(M) \to C^\infty(M)$. This automorphism originates from a diffeomorphism $\widetilde \varphi: M \to M$, i.e. $\alpha(f)=f\circ \widetilde\varphi$ for all $f\in C^\infty(M)$. Indeed this follows from the facts that ring automorphisms map maximal ideals to maximal ideals, and that the maximal ideals of $C^\infty(M)$ are of the form $I_p:=\{f\in C^\infty(M): f(p)=0\}$. See both \cite{MR2159792, MR2161810} and the references therein for more details.
 
  Hence for a.e. $p\in M$ and $f\in C^\infty(M)$, we have that $f(\varphi(p))=\alpha(f)(p)=f(\widetilde \varphi(p))$. This implies that $\varphi=\widetilde \varphi$ almost everywhere, and so in what follows we can simply assume that $\widetilde\varphi = \varphi$ from the beginning.

Now consider the isometry $ \eta:= \gamma \circ (\varphi^\ast)^{-1}: \mathcal M(M) \to \mathcal M(M)$. By the above rigidity and Remark~\ref{remark} this isometry is of the form $\overline \eta(\sigma)_p=\alpha_p(\sigma_p)$, where $\alpha_p:E_p \to E_p$, $p\in M$ have to be isometries. The family $\{\alpha_p\}_{p\in M}$ forms a section of the fibre bundle $\text{Isom}(E)$ with fibre the Lie group $G= \text{Isom}(C_0(P_1(n))) \cong \text{Isom}(P_1(n))$.   

 It remains to show that this section is smooth. To that end, take a local trivialization of the bundle $\overline E$, $(\pi_{\overline E})^{-1}(U) \to U \times C_0(P_1(n))$, $a \mapsto (\pi_{\overline E}(a), \psi_{\pi_{\overline E}(a)}(a))$, where for all $p\in U$, $\psi_p:(E_p,d_p) \to C_0(P_1(n))$ is a dilation. To show the smoothness of $\alpha$, we have to show that $\widetilde\alpha:U \to G$, $p\mapsto \psi_p\circ \alpha_p \circ \psi_p^{-1}\in G$ is smooth. 
 
There exist points $x_0,\dots,x_n\in P_1(n)$ such that the smooth map $\delta:G \to P_1(n)^{n+1}$, given by $\delta(g)=(g\cdot x_0,\dots,g\cdot x_n)$, is an immersion. To see this, fix $x_0 \in P_1(n)$ and an orthonormal frame $e_1,\dots,e_n$ of $T_{x_0}P_1(n)$, and define $x_i := \exp_{x_0}(e_i)$ for $i=1,\dots,n$. Since $P_1(n)$ is a CAT$(0)$ space, $t\mapsto \exp_{x_0}(te_i)$ is a minimizing geodesic. As isometries preserve minimizing geodesics, the condition $(g_1\cdot x_0,\dots, g_1\cdot x_n)=(g_2\cdot x_0,\dots, g_2\cdot x_n)$ implies $d_{x_0}g_1=d_{x_0}g_2$, which yields $g_1=g_2$; thus, $\delta$ is injective. Since $\delta$ is $G$-equivariant, it has constant rank, and therefore, by the constant rank theorem, the injective map $\delta$ is an immersion. 
 
By a partition of unity type argument, we can construct smooth metrics $g_0,g_1,...,g_n$ such that $\psi_p(g_i)=x_i$ for all $p\in U$ and  $i=0,1,\dots,n$. By our assumptions, $\alpha$ preserves the space of smooth Riemannian metrics and thus the map $U\to P(n)$, $p\mapsto (\psi_p \circ \alpha_p \circ \psi_p^{-1})(\psi_p(g_i))$, $i=0,1,...,n$ is smooth.  
 Thus the composition of $\widetilde\alpha$ and $\delta$ is smooth, i.e. $U\xrightarrow{\widetilde\alpha} G \xrightarrow{\delta} P_1(n)^{n+1}$ and therefore, since $\delta$ is an immersion we conclude that $\widetilde\alpha:U \to G$ is smooth also. Thus $\alpha \in \Gamma(\text{Isom}(E))$. 
 
 In other words, we have shown that $\gamma=\alpha \circ \varphi^\ast$, and so we deduce  $\text{Isom}(\mathcal M(M),d_E)= \Gamma(\text{Isom}(E))\cdot \text{Diff}(M)$. Further, we clearly have that $\Gamma(\text{Isom}(E))\cap \text{Diff}(M)$ is trivial. Finally we show that $\Gamma(\text{Isom}(E))$ is a normal subgroup of $\text{Isom}(\mathcal M(M),d_E)$: take any $\gamma\in \text{Isom}(\mathcal M(M),d_E)$ and $\beta \in \Gamma(\text{Isom}(E))$. By the above, there exist $\alpha \in \Gamma(\text{Isom}(E))$ and $ \varphi\in \text{Diff}(M)$ such that $\gamma=\alpha \varphi^\ast$. Clearly $\gamma  \beta  \gamma^{-1}=\alpha  \varphi^\ast  \beta  (\alpha  \varphi^\ast)^{-1}=\alpha  \varphi^\ast  \beta  (\varphi^{-1})^\ast \alpha$ and so by Remark~\ref{remark}, $$(\gamma  \beta  \gamma^{-1})(g)_p=(\alpha_p \tau^\varphi_p\beta_{\varphi(p)}\tau^{\varphi^{-1}}_{\varphi(p)}\circ \alpha_p)(g_p),$$ and we infer that $\gamma \beta \gamma^{-1} \in \Gamma(\text{Isom}(E))$. 
 Therefore, we conclude: $$\text{Isom}(\mathcal M(M),d_E)=\Gamma(\text{Isom}(E(M)))\rtimes\text{Diff}(M).$$
  \end{proof}

\begin{proof}[Proof of Theorem~\ref{main}]
	
	Every isometry $\gamma:(\mathcal{M}(M),d_E) \to (\mathcal{M}(N),d_E)$ uniquely extends to an isometry $\overline\gamma: (\mathcal E(M),d_{L^2}) \to (\mathcal E(N),d_{L^2})$. 
	
	By \cite{MR1308547}, any compact Riemannian manifold with the normalized volume measure is an atomless standard probability space. Therefore, there exists an isomorphism of probability spaces $\psi:(M,\mu_{g_0})\to (N,\mu_{h_0})$, where $\mu_{g_0}$ and $\mu_{h_0}$ are the volume measures with respect to fixed background metrics $g_0$ and $h_0$ of total volume one on $M$ and $N$ respectively. 
	
	Choose some family of isometries $\beta_p:\overline{E(N)}_{\psi(p)} \to \overline{E(M)}_p$, $p\in M$. By the change of variable formula, the map $\gamma_\psi:\mathcal{E}(N) \to \mathcal{E}(M)$, given by $\gamma_\psi(\sigma)_p=\beta_p(\sigma_{\psi(p)})$ is an isometry. 	
	
We then obtain the isometry $\gamma_\psi\circ \overline \gamma: \mathcal{E}_{L^2}(M) \to \mathcal{E}_{L^2}(M)$. By Theorem~A, we know that there exist an almost everywhere bijective $\varphi:M\to M$ and a family $\{\tau_p\}_{p\in M}$ of dilations $\tau_p:(\overline{E(M)}_{\varphi(p)},d_{\varphi(p)}) \to  (\overline{E(M)}_p,d_p)$ with positive factor, such that $(\gamma_\psi\circ\overline\gamma)(\sigma)_p=\tau_{p}(\sigma_{\varphi(p)})$. Therefore, $\overline \gamma(\sigma)_{p}=\kappa_p(\sigma_{\varphi(\psi^{-1}(p))})$, with dilations $\kappa_p:=\beta_{\psi^{-1}(p)}^{-1}\circ \tau_{\psi^{-1}(p)}:\overline{E(M)}_{\varphi(\psi^{-1}(p))} \to \overline{E(N)}_p$. 
	
	Now we run an almost identical argument as in the proof of Theorem~\ref{isom_completion}: for a smooth metric $g\in \mathcal{M}(M)$ and a positive function $f\in C^\infty_+(M)$, $f\cdot g \in \mathcal{M}(M)$ and $\overline \gamma(f\cdot g)=\hat f \cdot \overline \gamma(g)$, where $\hat f:N \to \R_{>0}$ is given a.e. by $\hat f(p)=f(\varphi(\psi^{-1}(p)))$. Since both $\overline\gamma(f\cdot g)$ and $\overline\gamma(g)$ lie in the space of smooth metrics $\mathcal{M}(N)$, we deduce that $\hat f\in C^\infty_+(N)$. Thus we obtain a map $C^\infty_+(M) \to C^\infty_+(N)$, $f \mapsto \hat f$ that respects products and sums. This extends to a ring isomorphism $C^\infty(M) \to C^\infty(N)$. By the same reasoning as above, this isomorphism is induced by a diffeomorphism $\widetilde \varphi:M \to N$. Thus $M$ and $N$ are diffeomorphic. 
\end{proof}

\small\section*{Acknowledgments} I would like to thank Nicola Cavallucci for insightful conversations on his research. I am grateful to Urs Lang for supervising my Master's thesis, which originally motivated my interest in these questions, and Alexander Lytchak for inspiring discussions and helpful suggestions. Finally I thank Maximilian Wackenhuth for discussions on this project. This
research was partially supported by the Deutsche Forschungsgemeinschaft (DFG, German Research Foundation) under project number 281869850.

\bibliographystyle{alpha}
\bibliography{paper}

\end{document}